\documentclass[11 pt]{amsart}
\usepackage{amsmath,amsthm,amsfonts,amssymb,amscd,amsrefs}
\usepackage[top=3cm, bottom=3cm,left=2cm, right=2cm]{geometry}
\usepackage{centernot}
  \usepackage[linkcolor=red,citecolor=red,colorlinks=true]{hyperref}
\usepackage{mathrsfs}
\usepackage{mathtools}
\usepackage{etoolbox}
\usepackage{scalerel,stackengine}
\stackMath
\newcommand\widecheck[1]{%
\savestack{\tmpbox}{\stretchto{%
  \scaleto{%
    \scalerel*[\widthof{\ensuremath{#1}}]{\kern-.6pt\bigwedge\kern-.6pt}%
    {\rule[-\textheight/2]{1ex}{\textheight}}
  }{\textheight}%

}{0.5ex}}%
\stackon[1pt]{#1}{\scalebox{-1}{\tmpbox}}%
}
\theoremstyle{definition}
\newtheorem{defn}{Definition}[section]

\newtheorem{eg}[defn]{Example}

\theoremstyle{remark}

\newtheorem{rmk}[defn]{Remark}
\theoremstyle{plain}
\newtheorem{thm}[defn]{Theorem}
\newtheorem{lem}[defn]{Lemma}
\newtheorem{cor}[defn]{Corollary}

\DeclareMathOperator*{\spn}{span}

\DeclareMathOperator*{\conv}{conv}

\numberwithin{equation}{section}
\makeatletter 
\newcommand{\rmnum}[1]{\romannumeral #1}
\newcommand{\Rmnum}[1]{\expandafter\@slowromancap\romannumeral #1@}

\makeatother
\makeatletter
\def\imod#1{\allowbreak\mkern10mu({\operator@font mod}\,\,#1)}
\makeatother
\makeatletter 
\patchcmd{\@settitle}{\uppercasenonmath\@title}{}{}{}
\patchcmd{\@setauthors}{\MakeUppercase}{}{}{}
\patchcmd{\section}{\scshape}{}{}{}
\makeatother
\begin{document}
\title{Ball proximinality of $M$-ideals of compact operators}
\author{C. R. Jayanarayanan}
\address{Department of Mathematics, Indian Institute of Technology Palakkad, 678557, India}
\email{crjayan@iitpkd.ac.in}
\author{Sreejith Siju}
\address{Department of Mathematics, Indian Institute of Technology Palakkad, 678557, India}
\email{sreejithsiju5@gmail.com}
\begin{abstract}
 In this article, we prove the proximinality of closed unit ball of $M$-ideals of compact operators. We also prove  the ball proximinality of $M$-embedded spaces in their biduals. Moreover, we  show that  $\mathcal{K}(\ell_1)$, the space  of compact operators on $\ell_1$, is ball proximinal  in  $\mathcal{B}(\ell_1)$, the space of bounded operators on $\ell_1$, even though $\mathcal{K}(\ell_1)$ is not an $M$-ideal  in $\mathcal{B}(\ell_1)$.
\end{abstract}
\maketitle
\section{Introduction}
One of the important problems which arises naturally in approximation theory is  the problem of existence of best approximation for a given vector from a specified subset of a Banach space. If this happens for every vector, then the subset is said to be proximinal. Clearly compact subsets of a Banach space are proximinal. However, lack of sufficient compact sets in infinite dimensional Banach spaces
makes the existence problem non-trivial. 
The present article studies the proximinality of closed unit ball of space of compact operators when it is an $M$-ideal in the space  of bounded linear operators.

Throughout this article,  $X$ will denote an infinite dimensional Banach space and  $H$ will denote an infinite dimensional Hilbert space. For a Banach space $X$; let $B_X$, $S_X$ and $B(x,r)$ denote the closed unit ball, the unit sphere and the closed ball with center at $x\in X$ and radius $r$ respectively. We consider every Banach space $X$, under the canonical embedding, as a subspace of $X^{**}$. For a Banach space $X$, let $\mathcal{B}(X)$ denote the space of all bounded linear operators on $X$, and $\mathcal{K}(X)$ denote the space of all compact operators on $X$. The essential norm $\|T\|_e$ of an operator $T\in \mathcal{B}(X)$ is the distance from $T$ to the space of compact operators, that is, $\|T\|_e = d(T, \mathcal{K}(X)) $.

We first recall some basic definitions and results which  will be needed later.    
\begin{defn}
	Let $K$ be a non-empty closed subset of a Banach space $X$. For $x \in X$, let  $P_K(x) = \{k\in K: d(x, K)= \|x-k\| \}$, where  $d(x, K)$ denotes the distance of $x$ from $K$.
	An element of $P_{K}(x)$ is called a best approximation to $x$ from $K .$ The set $K$ is said to be proximinal in $X$ if $P_{K}(x) \neq \phi$ for all $x \in X$. A closed subspace $Y$ of $X$ is said to be \emph{ball proximinal} in $X$ if  $B_Y$ is  proximinal in $X$ (see \cites{MR2374712} for details).
	\end{defn}

Clearly, ball proximinal subspaces are proximinal but the converse need not be true (see \cite{MR2374712}*{Proposition~2.4} and \cite{MR2146216}*{Theorem~1}).

In \cite{MR0493107}, Blatter raised the problem of identifying those Banach spaces $X$ which are proximinal in its bidual $X^{**}$, a problem which gained a lot of interest since the appearance of \cite{MR0493107}. In general, a Banach space need not be proximinal in its bidual (see \cites{MR0493107, MR0296659} for examples). In \cite{MR0296659}, Holmes and Kripke proved that for a Hilbert space $H$,  $\mathcal{K}(H)$ is proximinal in its bidual  $\mathcal{B}(H)$. They have also asked whether $\mathcal{K}(X)$ is proximinal in $\mathcal{B}(X)$ when $X$ is a Banach space. In general, $\mathcal{K}(X)$ need not be  proximinal in $\mathcal{B}(X)$. In the present paper, we discuss the ball proximinality of $\mathcal{K}(X)$ in $\mathcal{B}(X)$. Since  $\mathcal{K}(X)$ may not be even proximinal, we cannot expect ball proximinality of $\mathcal{K}(X)$ in general. So we discuss the ball proximinality of $\mathcal{K}(X)$ under some additional assumption on $X$. Towards this we recall the notion of an $M$-ideal  which is stronger than proximinality.

\begin{defn}[\cite{MR1238713}]
Let $X$ be a Banach space and $Y$ be a closed subspace of $X$. A linear projection $P$ on $X$ is said to be an \emph{$L$-projection} (\emph{$M$-projection}) if  $\|x\|=\|Px\|+\|x-Px\| \left(\|x\| = \operatorname{max}\{\|Px\|,\|x-Px\|\}\right)$ for all $x \in X$. A closed subspace $Y$ of $X$ is said to be an \emph{$M$-ideal}  in $X$ if annihilator  of $Y$ in $X^*$, denoted by $Y^\bot$, is the range of an $L$-projection in $X^*$.  If $X$ is an $M$-ideal in $X^{**}$, then $X$ is said to be an $M$-embedded space. 
\end{defn}
It is well known that $\mathcal{K}(H)$ is an $M$-ideal in $\mathcal{B}(H)$ and $\mathcal{K}(\ell_p)$ is an $M$-ideal in   $\mathcal{B}(\ell_p)$ for $1<p<\infty$ (see \cites{MR1238713}). Hence $\mathcal{K}(H)$ and $\mathcal{K}(\ell_p)$ ($1<p<\infty$) are proximinal in their respective biduals. By \cite{MR1238713}*{Chapter~\Rmnum{6}, Proposition~4.11}, we know that if $X$ is reflexive and $\mathcal{K}(X)$ is an $M$-ideal in $\mathcal{B}(X)$, then $\mathcal{K}(X)^{**}=\mathcal{B}(X)$ and hence $\mathcal{K}(X)$ is proximinal in its bidual. However, $M$-ideals need not be ball proximinal (see \cite{MR3314889}). So it is natural to ask whether an $M$-embedded space is ball proximinal in its bidual. In this article, we give an affirmative answer to this question and also prove the ball proximinality of $\mathcal{K}(X)$ when $\mathcal{K}(X)$ is an $ M$-ideal in $\mathcal{B}(X)$.

In Section~\ref{hilbert}, we give an operator theoretic proof of ball proximinality of $\mathcal{K}(H)$ in $\mathcal{B}(H)$. We also prove the distance formula, $d(T,B_{\mathcal{K}(H)}) = \operatorname{max}\left\{\|T\|-1, d(T, \mathcal{K}(H))\right\}$ for $T\in \mathcal{B}(H).$

Another approach  to study proximinality of space of compact operators is by using the following notion of basic inequality, introduced by Axler, Berg, Jewell, and Shields (see \cites{MR1971228, MR576869} for details). 
Recall that a net $(A_\alpha)$ in $\mathcal{B}(X)$ converges to $A\in \mathcal{B}(X)$ in strong operator topology (SOT) if $\|A_\alpha x-Ax\| \longrightarrow 0$ for every  $x\in X$, and  $(A_\alpha)$ converges to $A$ in weak operator topology (WOT) if $x^{*}(A_\alpha x)\longrightarrow 0$ for every $x\in X$ and $x^*\in X^*.$
\begin{defn}[\cite{MR576869}]
A Banach space $X$ is said to satisfy the basic inequality if for each $T \in \mathcal{B}(X)$ and each bounded net $\left(A_{\alpha}\right)$ in $\mathcal{K}(X)$ such that $A_{\alpha} \rightarrow 0$ in SOT
 and $A_{\alpha}^{*} \rightarrow 0$ in SOT, the following is true: for each $\varepsilon>0$, there exists an index $\beta$ such that
 \[
 \left\|T+A_{\beta}\right\| \leq \varepsilon+\max \left(\|T\|,\|T\|_{e}+\left\|A_\beta\right\|\right).
 \]	
 \end{defn}

 The basic inequality was originally defined in~\cite{MR1971228} where sequences were used instead of nets (see   Section~2 of \cite{MR576869} to see the consequence of replacing sequences by nets in the definition of basic inequality).

In Section~\ref{banach}, we prove the ball proximinality of $\mathcal{K}(X)$ in $\mathcal{B}(X)$ when $\mathcal{K}(X)$ is an $M$-ideal in $\mathcal{B}(X)$. The key idea for proving this result is the inequality mentioned in Lemma~\ref{2}. This inequality is basically a modification of the revised basic inequality discussed in \cite{MR1257062}. In fact,  inequality in Lemma~\ref{2} and revised basic inequality  in \cite{MR1257062} are modified versions of the basic inequality. We  also prove the ball proximinality of $M$-embedded spaces  in their  biduals. 
Moreover, we  show that  $\mathcal{K}(\ell_1)$ is ball proximinal  in $\mathcal{B}(\ell_1)$ even though $\mathcal{K}(\ell_1)$ is not an $M$-ideal  in $\mathcal{B}(\ell_1)$.

  \section{Ball proximinality of space of compact operators on Hilbert space.}
  \label{hilbert}
In this section, we prove the ball proximinality of $\mathcal{K}(H)$ in $\mathcal{B}(H)$ when $H$ is a Hilbert space. 

Let $ S(H)$ be the set of all  sequences of unit vectors in the Hilbert space $H$ which converge weakly to $0$. For $T\in \mathcal{B}(H)$, define  $\Delta(T)= \sup$ $\{\displaystyle{\limsup_{n\to\infty}\|Ts_n\|: (s_n)\in S(H) }\}$ as in \cite{MR0296659}. In \cite{MR0296659}*{Section 3,  Theorem}, it is proved that $ d(T,\mathcal{K}(H))=\Delta(T)$. The following theorem derives  an analogous distance formula for $B_{\mathcal{K}(H)}$ and establishes the ball proximinality of $\mathcal{K}(H)$ in $\mathcal{B}(H)$.
\begin{thm}\label{Hcase}
	Let $H$ be an infinite dimensional Hilbert space. Then 
	\begin{enumerate}
		\item $d(T,B_{\mathcal{K}(H)}) = \max \left\{\|T\|-1,\Delta(T)\right\}=\max \left\{\|T\|-1,d(T,\mathcal{K}(H))\right\}$ for  $T\in \mathcal{B}(H)$ and\label{dist}
		\item $\mathcal{K}(H)$ is ball proximinal in $\mathcal{B}(H)$.
	\end{enumerate}
\end{thm} 
  \begin{proof} 
Let $T\in \mathcal{B}(H)$. Since, by \cite{MR0296659}, $d(T,\mathcal{K}(H))=\Delta(T)$; it follows that $\max\left\{\|T\|-1,\Delta(T)\right\} \leq d(T,B_{\mathcal{K}(H)})$. 
Without loss of generality, we may assume that $T$ is not a compact operator. 
For, if $T$ is a compact operator then $\frac{T}{\|T\|}$ is a best approximation to $T$ from $B_{\mathcal{K}(H)}$.

\noindent {\bf Case 1}: $\|T\|>1$.

Suppose $T$ does not attain its norm (that is, $\|T\|\ne\|Tx\|$ for every $x\in S_H$). Let $(s_n)$ be a sequence in $S_H$  such that $\lim_{n\rightarrow\infty} \|Ts_n\|= \|T\|$. Then, by \cite{MR0296659}*{Lemma 2}, $(s_n)\in S(H)$. Hence $\|T\| \leq \Delta(T) \leq d(T,B_{\mathcal{K}(H)})$. Thus $0$ is a best approximation to $T$ from $B_{\mathcal{K}(H)}$.

Now let $T \in \mathcal{B}(H)$ be a norm attaining operator (that is, there exists an $x\in S_H$ such that $\|T\|=\|Tx\|$). We follow  a technique similar to the one  used in Case~3 of \cite{MR0296659}*{Section 3,  Theorem} to construct an orthonormal sequence (finite or infinite) which will be used to define the best compact approximant from the closed unit ball of $\mathcal{K}(H)$.

Let $e_1 \in H $ be a unit vector such that $\|T(e_1)\| = \|T\|$ and $P_1$ be the projection onto $\{e_1\}^\bot$. 
Having chosen the orthonormal set $\{e_1, e_2, ..., e_n\}$, let $P_n$ be the orthogonal projection onto $\{e_1, e_2, ..., e_n\}^\bot=E_n$. Let $E_0=H$ and $P_0$ be the identity operator on $H$. Since $T$ is not compact and $I-P_n$ is a finite rank projection, $\|TP_n\|\neq0$. If $TP_n$ attains its norm, then choose a unit vector $e_{n+1}$  such that $\|TP_n(e_{n+1})\|= \|TP_n\|$. If $TP_n$ does not attain its norm, then we stop the process of choosing unit vectors and proceed with the finite set $\{e_1,\ldots,e_n\}$ to obtain a best approximation as we shall see  in  Subcase~1 below. On the other hand, if $TP_n$ attains its norm at $e_{n+1}$ for all $n$, then we proceed with the infinite set $\{e_1,e_2,\ldots\}$ as we shall see  in Subcase~2 below.

\noindent {\bf Subcase 1:} Suppose that $P_m$ exists only for finitely many $m$. That is, there exists an $n\ge 1$ such that $TP_{k-1}$  attains its norm  at $e_k$ for $k = 1,2,...,n$ and $TP_n$ does not attain its norm.

In this case, by a similar argument as in the proof of  Case~3 of \cite{MR0296659}*{Section 3,  Theorem}, we get that
$P_{k-1}\left(e_{k}\right)=e_{k}$,\, $e_{k} \in E_{k-1}$ and 
$\langle T(e_k),TP_k(x)\rangle = 0$ for all $x\in H$ and $k = 1,2,...,n$.

In particular, for  $1\leq i< j\leq n$, $\langle T(e_i), T(e_j)\rangle = \langle T(e_i), TP_i(e_j)\rangle=0$  and 
$\langle T(e_k), TP_n(x)\rangle = \langle T(e_k), TP_k(P_n(x))\rangle = 0  \mbox{ for } k = 1,2,..., n.$

Now define an operator $L: H\to H$ as
$$L= \frac{1}{\|T\|}\sum_{i=1}^n Te_i\otimes e_i, $$
where $(Te_i\otimes e_i)(x) =  \langle x,e_i\rangle Te_i$.  

Since $\langle Lx, Lx\rangle = \frac{1}{\|T\|^2}\sum_{i=1}^n |\langle x,e_i \rangle |^2\|Te_i\|^2\leq \|x\|^2$ for every $x\in H$, we have $\|L\|\leq 1$.

Let $x\in B_H$. Write $ x= u+v$, where $u\in \spn\{e_1,e_2,..., e_n\}$ and $v\in $ $\{e_1,e_2,..., e_n\}^\bot$. Then $Lv=0$ and 
\begin{align*}
\langle (T-L)u, Tv\rangle & = \left\langle(T-L)(\textstyle{\sum_{i=1}^{n}\langle u, e_i\rangle e_i )}, TP_nv\right\rangle\\
&= \left\langle \textstyle{\sum_{i=1}^{n}(1 - \frac{1}{\|T\|})\langle u, e_i\rangle Te_i},  TP_nv\right\rangle\\
&=\left(1- \frac{1}{\|T\|}\right) \textstyle{\sum_{i=1}^{n}\langle u, e_i\rangle \langle Te_i}, TP_nv\rangle = 0.
\end{align*}
 In addition, by Case~2 of \cite{MR0296659}*{Section 3,  Theorem}, we have $\|TP_n\| = \Delta(T)$. Therefore,
\begin{align*}
\|(T-L)(x)\|^2 &=\langle (T-L)(u+v), (T-L)(u+v)\rangle\\
& =   {\langle (T-L)u, (T-L)u\rangle+  \langle Tv, Tv \rangle }\\
& =  \left\langle \sum_{i=1}^{n}\left(\frac{\|T\|-1}{\|T\|}\right)\langle u, e_i\rangle Te_i, \sum_{i=1}^{n}\left(\frac{\|T\|-1}{\|T\|}\right)\langle u, e_i\rangle Te_i\right\rangle + \langle TP_nv, TP_nv \rangle \\
&=\left(\frac{\|T\|-1}{\|T\|}\right)^2\sum_{i=1}^n|\langle u, e_i \rangle|^2 \|Te_i\|^2 + \|TP_nv\|^2\\
&\leq \left(\|T\|-1\right)^2\|u\|^2 + \Delta(T)^2\|v\|^2\\
&\leq \max \{(\|T\|-1)^2, \Delta(T)^2\}\|x\|^2.
\end{align*}
Thus $\max\left\{\|T\|-1,\Delta(T)\right\} \leq  d(T,B_{\mathcal{K}(H)}) \leq \|T-L\|\leq$  $\max\left\{\|T\|-1,\Delta(T)\right\}$. So $d(T,B_{\mathcal{K}(H)}) = \|T-L\|$. Hence $L$ is a best approximation to $T$ from $B_{\mathcal{K}(H)}$.

\noindent{\bf Subcase 2:}
Suppose $P_n$ exists for all $n\in \mathbb{N}$. That is, $TP_n$ attains its norm at $e_{n+1}$ for all $n\in \mathbb{N}$.


In this case, by a similar argument as in Case~3 of \cite{MR0296659}*{Section 3,  Theorem},  we can see that the  orthonormal sequence of vectors $(e_n)_{n=1}^\infty$ satisfies the following properties.
\begin{align}
P_n(e_{n+1}) &=e_{n+1} \mbox{ and hence } \|TP_n\|=\|Te_{n+1}\|  \mbox{ for all } n\in\mathbb{N},\label{ineq1}\\
\left\langle T(e_n),TP_n(x)\right\rangle &= 0 \hspace{6pt}\textnormal{for all  $n\in \mathbb{N}$ and  for any $x\in H$},\label{inq2}\\
\langle T(e_i), T(e_j)\rangle &= 0 \hspace{10pt} \textnormal{for all $i\neq j$,} \\\label{inq3}
\lim_{n\rightarrow\infty}\|Te_n\| &= \Delta(T) \hspace{10pt} \textnormal{and}\\
\Delta(T)&\leq \|Te_n\| \mbox{ for all } n.\label{inq5}
\end{align}

Since $T$ is not compact, $\Delta(T) >0 $ and hence, by (\ref{inq5}), $\|Te_n\| > 0$ for all $n$. Now define an operator $L: H \longrightarrow H$ as 
\begin{equation*}
L = \frac{1}{\|T\|}{\sum_{n=1}^K Te_n\otimes {e_n}}+\sum_{n=K+1}^{\infty}{\left(\|Te_n\|-\Delta(T)\right)}\frac{Te_n}{{\|Te_n\|}}\otimes {e_n},
\end{equation*}
where $K$ is chosen so that $\left\|Te_n\right\|-\Delta(T) \leq 1$ for all $n\geq K$. 
Since the sequence $\left(\|Te_n\|-\Delta(T)\right)$ converges to 0 and $(\frac{Te_n}{\|Te_n\|})$ forms an orthonormal sequence, by \cite{MR0257800}*{Page 8, Theorem 1} and  \cite{MR0257800}*{Page 13, Corollary}, $L$ is a compact operator with $\|L\|\leq1$. 

Let $x\in B_H$. Write $x=u+v$, where $u\in \spn\{e_1,e_2,...,\}$ and $v\in \{e_1,e_2,....\}^\bot$.  
Define
\begin{equation*}
\lambda_n=
\begin{cases}
\frac{\|T\|-1}{\|T\|}\langle u, e_n \rangle & \text{if $n\leq K$},\\
\frac{\Delta(T)}{\|Te_n\|}\langle u,e_n \rangle & \text{if $n> K$}. 
\end{cases}
\end{equation*}
Then $(T-L)u=(T-L)(\sum_{n=1}^{\infty}\langle u,e_n \rangle e_n)=\sum_{n=1}^{\infty}\lambda_nTe_n$. Since $v\in \{e_1,e_2,\ldots\}^\bot$,  $P_nv = v$ for all $n\in \mathbb{N}$. 
Now, using (\ref{inq2}), we get 
 \begin{align}
 \label{computation}
 \left\langle Tv,(T-L)u\right\rangle 
 =\left\langle Tv, \sum_{n=1}^{\infty}\lambda_nTe_n\right\rangle
 =\sum_{n=1}^{\infty}\left\langle Tv, \lambda_nTe_n\right\rangle
 = \sum_{n=1}^{\infty}\left\langle TP_nv,\lambda_nTe_n\right\rangle
 =0.
  \end{align} 
Since
 $\|Tv\|^2 =\|TP_nv\|^2
 \leq \|TP_n\|^2\|v\|^2
 \leq \|Te_{n+1}\|^2\|v\|^2$ for every $n$, by (\ref{inq3}), we get 
 $\|Tv\|^2\leq \Delta(T)^2\|v\|^2$. 
 This together with (\ref{computation}) and using the fact that $Lv = 0$,  we get
 \begin{align*}
 \|(T-L)x\|^2&=\left\langle (T-L)u,(T-L)u\right\rangle + 2Re\left\langle Tv,(T-L)u\right\rangle+\left\langle Tv, Tv\right\rangle\\
 &= \left\langle\sum_{n=1}^K \lambda_{n} Te_n +\sum_{n=K+1}^\infty \lambda_nTe_n,\sum_{n=1}^K \lambda_n Te_n+\sum_{n=K+1}^\infty \lambda_n Te_n \right\rangle + \|Tv\|^2\\
 &\leq   \left\langle\sum_{n=1}^K \lambda_{n} Te_n, \sum_{n=1}^K \lambda_{n} Te_n\right\rangle +\left\langle\ \sum_{n=K+1}^\infty \lambda_n Te_n, \sum_{n=K+1}^\infty \lambda_n Te_n\right\rangle  +  \Delta(T)^2\|v\|^2\\
 &\leq\left({\|T\|-1}\right)^2\sum_{n=1}^K| \langle u,e_n\rangle|^2+\sum_{n=K+1}^\infty \Delta(T)^2|\langle u,e_n\rangle|^2+\Delta(T)^2\|v\|^2\\
 &\leq  \max\left\{(\|T\|-1)^2,\Delta(T)^2\right\}\|x\|^2.
 \end{align*}

Thus $d(T,B_{\mathcal{K}(H)})=\|T-L\|=\max\left\{\|T\|-1, \Delta(T)\right\}$. Hence $L$ is a best approximation to $T$ from $B_{\mathcal{K}(H)}$.

\noindent{\bf Case~2:} $\|T\| \leq 1$. 

Since $T$ is not compact, by \cite{MR0296659}*{Section 3, Theorem}, we can choose an element $L\in P_{\mathcal{K}(H)}(T)$ such that $L$ is either of the form   $L=\sum_{n=1}^{\infty}\left(\left\|T\left(e_{n}\right)\right\|-\Delta(T)\right) \frac{Te_n}{{\|Te_n\|}}\otimes {e}_{n}$ for some orthonormal sequence $(e_n)$ with $\|T(e_n)\|>0$ for all $n$ or $L=TP$ for some finite rank  orthogonal projection $P$ on $H$. Since $\|T\|\leq 1$, by \cite{MR0257800}*{Page~8, Theorem~1}, it follows that  $\|L\|\leq 1$. Then $ d(T,B_{\mathcal{K}(H)})\leq \|T-L\|= d(T,\mathcal{K}(H))\leq d(T,B_{\mathcal{K}(H)})$. Thus $d(T,B_{\mathcal{K}(H)})=\max\left\{\|T\|-1, \Delta(T)\right\}=\|T-L\|$. Thus $L$ is a best approximation to $T$ from $B_{\mathcal{K}(H)}$ and hence  $\mathcal{K}(H)$ is ball proximinal in $\mathcal{B}(H)$.
\end{proof}
We now prove that the distance  $d(T, B_{\mathcal{K}(H)})$ coincides with $d(T, \mathcal{K}(H))$ when $T$ is a scalar multiple of an extreme point of closed unit ball of $\mathcal{B}(H)$.  Towards this, we describe $P_{B_X}(x)$ when $x$ is a scalar multiple of an extreme point of the closed unit ball of a Banach space $X$.

\begin{thm}\label{isometry}
	Let $e$ be an extreme point of the closed unit ball of a Banach space $X$. Then $P_{B_X}(\alpha e) = \{\frac{\alpha e}{|\alpha|} \}$ for all $\alpha \in \mathbb{C}$ with $|\alpha| >1$.
\end{thm}
\begin{proof}
 Assume first that $\alpha \in \mathbb{R}$ and  $\alpha >1$. Let $e$ be an extreme point of $B_X$. Then $-e$ is an extreme point of $B_X$. 
So $(1-\alpha)e$ is an extreme point of $B(0, \alpha -1)$. 
By shifting the center to $\alpha e$ we get that $e$ is an extreme point of $ B(\alpha e, \alpha -1).$    
Now suppose that there exists an $f\in B_X$ such that $\|\alpha e -f\|= \alpha -1=d(\alpha e, B_X )$.

Since $e, f \in  B(\alpha e, \alpha -1)$, we have $te+(1-t)f\in B(\alpha e, \alpha -1)$ for all $t\in [0, 1]$. 
Let $y= (2-t)e-(1-t)f$, where $t\in (0,1)$ is chosen so that $\alpha + t >2$. Then 
$$
\|\alpha e - ((2-t)e-(1-t)f)\|\leq |\alpha - (2-t)|+|1-t|
\leq |\alpha +t-2|+1-t
\leq \alpha - 1. 
$$
Hence $y\in B(\alpha e, \alpha -1)$. Moreover,  $\frac{y+te+(1-t)f}{2} = \frac{(2-t)e-(1-t)f+te+(1-t)f}{2} = \frac{2e}{2}= e$. 
Since $e$ is an extreme point of $B(\alpha e, \alpha -1)$, we must have $y =te+(1-t)f$. That is, $(2-t)e-(1-t)f = te+(1-t)f.$ 
Hence $e = f.$    

For $|\alpha| > 1$, if we let $f=\frac{\alpha e}{|\alpha|}$, then $f$ is again an extreme point of $B_X$. 
Hence by the first part of the proof we get  $P_{B_X}(\alpha e) = P_{B_X}(|\alpha|f)= \{f\}= \{\frac{\alpha e}{|\alpha|}\}.$
\end{proof}
 For a Hilbert space $H$, it is well known that the extreme points of the closed unit ball of $\mathcal{B}(H)$ are precisely the isometries and co-isometries. Now the distance formula  (\ref{dist}) in Theorem~\ref{Hcase} together with  Theorem~\ref{isometry} gives the following result.
\begin{cor}\label{isocor}
	Let $H$ be an infinite dimensional Hilbert space and $V\in \mathcal{B}(H)$ be an isometry or a co-isometry. Then for each $a\in \mathbb{C}$, $d(aV,B_{\mathcal{K}(H)})=d(aV, \mathcal{K}(H))$. 
\end{cor} 
\begin{proof}
If $|a|\leq 1$, then by Theorem \ref{Hcase}, $d(aV,B_{\mathcal{K}(H)})=d(aV, \mathcal{K}(H))$. Now suppose that $|a|> 1$ and  $d(aV,B_{\mathcal{K}(H)})\ne d(aV, \mathcal{K}(H))$. Then, by Theorem~\ref{Hcase}, there exists a $K\in B_{\mathcal{K}(H)}$ such that 
 $d(aV, B_{\mathcal{K}(H)})=\|aV-K\|=|a|-1=d(aV, B_{\mathcal{B}(H)})$. Hence, by Theorem~\ref{isometry}, $K=\frac{aV}{|a|}$, which is a contradiction since $V$ is not a compact operator. 
\end{proof}

\section{Ball proximinality of M-ideals of compact operators and M-embedded spaces.}
\label{banach}
In this section,  we prove the ball proximinality of $\mathcal{K}(X)$ in $\mathcal{B}(X)$ when $X$ is a Banach space such that $\mathcal{K}(X)$ is an $M$-ideal in $\mathcal{B}(X)$. In this direction, we begin  with a lemma which is a modification of  \cite{MR1257062}*{Proposition~2.3}. 

Note that if $Y$ is an $M$-ideal in $X$, then every $y^*\in Y^*$ has a unique norm preserving extension to a functional $x^*\in X^*$. Thus we can consider $Y^*$ as a subspace of $X^*$. So it make sense to define  the weak topology on $X$ induced by $Y^*$. We will denote this by $\sigma(X,Y^*)$.
\begin{lem}\label{lemmma}
	Let $J$ be an $M$-ideal in a Banach space $X$ and $x\in X$. Then there exists a net $(y_\alpha)$ in $J$ such that $(y_\alpha)$ converges to $\frac{x}{\|x\|}$ in the $\sigma(X,J^*)$-topology, and for  each $z\in X$ and $\varepsilon>0$ there exists an index $\alpha_0$ such that
		\begin{equation}\label{genineq}
	\|z+\beta(x-y_{\alpha})\|\leq  \varepsilon + \max\{\|z+\frac{\|x\|-1}{\|x\|}\beta x\|, \|z+J\|+\beta\|x+J\|\}\mbox{ for every }\alpha\geq\alpha_0 \mbox{ and }   \beta\in [0,1].
	\end{equation}
	Consequently, 
	\begin{equation*}
	\limsup_\alpha\|z+\beta(x-y_\alpha)\|\leq  \max\{\|z+\frac{\|x\|-1}{\|x\|}\beta x\|, \|z+J\|+\beta\|x+J\|\}\mbox{ for every } z\in X \mbox{ and } \beta\in [0,1].
	\end{equation*}
\end{lem}
\begin{proof}
We follow the proof technique of \cite{MR1257062}*{Proposition~2.3}. Let $Q$ denote the $M$-projection from $X^{**}$ onto $J^{\bot\bot}$. Then 
\begin{align*}
\|z+x-\frac{Qx}{\|x\|}\|
&=  \|Q(z+\frac{\|x\|-1}{\|x\|}x)+(I-Q)(z+x)\|\\
&\leq \max\{\|(z+\frac{\|x\|-1}{\|x\|}x)\|, \|z+J\|+\|x+J\|\} \mbox{ for every } z\in X.
\end{align*}

Let $A$ be the set of all triplets $\alpha=(E, F, \varepsilon)$ where $E\subset X^{**}$ and $F\subset X^*$ are finite dimensional subspaces and $\varepsilon >0$. Then $A$ can be partially ordered as: $(E_1, F_1, \varepsilon_1)\leq (E_2, F_2, \varepsilon_2)$ if $E_1\subseteq E_2$, $F_1\subseteq F_2$ and $\varepsilon_2\leq \varepsilon_1$. Then corresponding to each triplet $\alpha$, by principle of local reflexivity (\cite{MR1121711}*{Theorem 3.2}), there exists an operator $T_\alpha$ such that 
$\|T_\alpha\| \leqslant(1+\varepsilon)$, $T_\alpha|_{E \cap X}=\mathrm{Id}$, $T_\alpha\left(E \cap J^{\bot\bot}\right) \subset J $ and $\left\langle T_\alpha x^{* *}, x^{*}\right\rangle=\left\langle x^{* *}, x^{*}\right\rangle$  for every $x^{* *} \in E$ and $x^{*} \in F$. 
 Define $y_\alpha = T_\alpha(\frac{Qx}{\|x\|})$.
Then, by proceeding as in the proof of \cite{MR1257062}*{Proposition~2.3}, we can see that $y_{\alpha} \rightarrow \frac{x}{\|x\|}$ in the $\sigma\left(X, J^*\right)$-topology.
Assume now that  $\varepsilon > 0 $ be given and $z\in X$.  Write $\delta = \max\left\{\|z\|+\left|\|x\|-1\right|, \|z+J\|+\|x+J\|\right\}$. Let $\alpha_0 = (E_0, F_0, \frac{\varepsilon}{\delta})$, where $E_0$ and $F_0$ are finite dimensional subspaces of $X^{**}$ and $X^*$ respectively such that $z,x\in E_0$.

Now for any $1\geq \beta > 0$ and $\alpha \geq \alpha_0$, 
\begin{align*}
\left\|\frac{z}{\beta}+x-y_{\alpha}\right\|& = \left\|T_{\alpha}(\frac{z}{\beta}+x-\frac{Qx}{\|x\|})\right\|
\leq \left(1+\frac{\varepsilon}{\delta}\right)\left\|\frac{z}{\beta}+x-\frac{Qx}{\|x\|}\right\|\\
& \leq \left(1+\frac{\varepsilon}{\delta}\right)\operatorname{max}\left\{\left\|\frac{z}{\beta}+\frac{\|x\|-1}{\|x\|}x\right\|, \frac{1}{\beta}\|z+J\|+\|x+J\|\right\} \\
&\leq \operatorname{max}\left\{\left\|\frac{z}{\beta}+\frac{\|x\|-1}{\|x\|}x\right\|, \frac{1}{\beta}\|z+J\|+\|x+J\|\right\} + \frac{\varepsilon}{\beta\delta}\delta.
\end{align*}
Therefore, 
\begin{align*}
\|z+\beta(x-y_{\alpha})\| = \beta\|\frac{z}{\beta}+x-y_{\alpha}\| \leq \varepsilon + \operatorname{max}\left\{\|{z}+\frac{\|x\|-1}{\|x\|}\beta x\|, \|z+J\|+\beta\|x+J\|\right\} 
\end{align*}
for every $ \alpha\geq \alpha_0$ and $0<\beta\leq 1$.
\end{proof}

Now by applying  Lemma~\ref{lemmma} to $M$-ideals of compact operators and by following the arguments used in the proof of  (\rmnum{1})$\implies$ (\rmnum{2}) of \cite{MR1257062}*{Theorem~3.1}, we obtain the inequality in the following lemma which is basically a modification of the revised basic inequality studied in \cite{MR1257062}. 
\begin{lem}\label{2}
Let $X$ be a Banach space such that $\mathcal{K}(X)$ is an $M$-ideal in $\mathcal{B}(X)$ and let $T\in \mathcal{B}(X)$. Then there exists a net $(L_\alpha)$ in $\mathcal{K}(X)$  such that $L_{\alpha}^* \longrightarrow \frac{T^*}{\|T\|}$ in  WOT,  and for each $S\in \mathcal{B}(X)$ and $\varepsilon >0$ there exists an $\alpha_0$ such that  
	\begin{equation}\label{eq1}
	\|S+\beta (T-L_{\alpha})\|\leq \varepsilon +\operatorname{max}\left\{\left\|S+\frac{(\|T\|-1)}{\|T\|}\beta T\right\|, \|S\|_e+\beta\|T\|_e\right\}
	\end{equation}
	for every  $\alpha \geq \alpha_0$ and $0<\beta\leq 1.$ Consequently, 
	\begin{equation*}
	\limsup_\alpha \|S+\beta(T-L_\alpha)\|\leq \operatorname{max}\left\{\left\|S+\frac{(\|T\|-1)}{\|T\|}\beta T\right\|, \|S\|_e+\beta\|T\|_e\right\}\mbox{ for all } S\in \mathcal{B}(X) \mbox{ and } \beta\in [0, 1].
	\end{equation*} 
\end{lem}

We now prove the main theorem of this section. We prove the ball proximinality of $\mathcal{K}(X)$ when it is an $M$-ideal in $\mathcal{B}(X)$. For an operator $T$ with $\|T\|>1$,  we construct, using the basic inequality method,   a sequence $(\lambda_i)$ of positive numbers and a sequence $(T_{\alpha(i)})$  of compact operators such that  $K= \sum_{i=1}^{\infty}\lambda_iT_{\alpha(i)}$ is a best approximation to $T$ from the closed unit ball of $\mathcal{K}(X)$.

\begin{thm}\label{thmmain}
	Let $X$ be a Banach space such that $\mathcal{K}(X)$ is an $M$-ideal in $\mathcal{B}(X)$. Then $\mathcal{K}(X)$ is ball proximinal in $\mathcal{B}(X)$. Moreover, $d(T, B_{\mathcal{K}(X)}) = \operatorname{max}\left\{\left\|T\right\|-1, d(T, {\mathcal{K}(X)})\right\}$ for all $T\in \mathcal{B}(X).$
\end{thm}
\begin{proof}
Let $T\in \mathcal{B}(X)$. Without loss of generality, we may assume that $T$ is not a compact operator. For, if $T$ is a compact operator then $\frac{T}{\|T\|}$ is a best approximation to $T$ from $B_{\mathcal{K}(X)}$.

Since $\mathcal{K}(X)$ is an $M$-ideal in $\mathcal{B}(X)$, by \cite{MR1238713}*{Chapter~\Rmnum{6}, Proposition~4.10}, there exists a net $(K_\alpha)$ of compact operators  with $\|{K_\alpha}\| \leq 1 $ and  $(K_\alpha^*)$ converges to $\frac{T^{*}}{\|T\|}$ in SOT. Let $(L_\alpha)$ be the net devised from Lemma~\ref{2}. Then $(K_\alpha^* - L_\alpha^*)$ converges to 0 in WOT. Hence $(K_\alpha - L_\alpha)$ converges to 0 in the weak topology  $\sigma\left(\mathcal{K}(X), \mathcal{K}(X)^{*}\right)$. Thus $\|{T}_{\alpha}-{S}_{\alpha}\|\to 0$ for some ${T}_{\alpha} \in \conv\{K_{\beta} | \beta \geqslant \alpha\}$ and ${S}_{\alpha} \in \conv\{L_{\beta} | \beta \geqslant \alpha\}$, where $\conv$ denotes the convex hull. Since $({S}_{\alpha})$ satisfies inequality (\ref{eq1}),  the net $({T}_{\alpha})$ of compact operators with $\|T_\alpha\|\leq 1$ also satisfies inequality (\ref{eq1}).

 Suppose $\|T\|> 1$.  Recall that $d(T, {\mathcal{K}(X)}) = \|T\|_e$. Let $0< a < 1$ be such that  $a\|T\|-1 > 0$. Pick  a strictly decreasing sequence $(\varepsilon_{i})$ of positive numbers such that   $\sum_{i=1}^{\infty}\varepsilon_i$ converges and 
$a \leq \frac{\operatorname{max}\left\{\left\|T\right\|-1,\|T\|_e\right\}-\sum_{i=1}^{\infty}\varepsilon_i}{\operatorname{max}\left\{\left\|T\right\|-1,\|T\|_e\right\}}$.

Now by using a  technique similar to the one used in \cite{MR1971228}*{Theorem~1},  we inductively construct a sequence $(\lambda_i)$ of positive numbers and a sequence $(T_{\alpha(i)})$  of compact operators from the net $(T_\alpha)$  such that 
\begin{align}
\sum_{i=1}^{\infty}\lambda_i  \leq & 1, \label{lambda}\\
\left \|\sum_{i=1}^{n}\lambda_i(T-T_{\alpha(i)})\right\|  = & \operatorname{max}\left\{\|T\|-1,\|T\|_e\right\}-\varepsilon_n \mbox{ for  } n = 1, 2,\ldots. \label{12} 
\end{align}
\begin{subequations}
	\begin{align}
	&\|\lambda_1(T-T_{\alpha(1)})\|\leq \varepsilon_1+ \operatorname{max}\{\lambda_1\left(\|T\|-1\right), \lambda_1\|T\|_e\},\label{b}\\
	&\left\|\sum_{i=1}^{n+1}\lambda_i(T-T_{\alpha(i)})\right\| \leq  \varepsilon_{n+1}+ \operatorname{max}\left\{\left\|  \sum_{i=1}^{n}\lambda_i(T-T_{\alpha(i)}) +\frac {\lambda_{n+1}(\|T\|-1)T}{\|T\|}\right\|,  (\sum_{i=1}^{n+1}\lambda_i)\left\| T\right\|_e\right\} \label{ineqaulitymain}	
\end{align} for $n =1, 2, \ldots$.
\end{subequations}

We start the construction by taking  $S= 0 $ in Lemma~\ref{2}. Then there exists an $\alpha(1)$ such that 
$$\left\|\beta(T-T_{\alpha(1)})\right\| \leq \varepsilon_1+\operatorname{max}\left\{\left\|\frac{(\|T\|-1)}{\|T\|}\beta T\right\|, \beta\|T\|_e \right\} \mbox{ for all }\beta\in [0, 1].$$ 

Now let $\lambda_1>0$ be such that $\|\lambda_1(T-T_{\alpha(1)})\| = \operatorname{max}\left\{\|T\|-1,\|T\|_e\right\}-\varepsilon_1$.

Suppose we have chosen scalars $\lambda_1, \lambda_2,..., \lambda_n$ and indices $\alpha(1), \alpha(2),..., \alpha(n)$ such that $\alpha(1)<\alpha(2)<...<\alpha(n)$ and 
\begin{equation}\label{ind12}
\left \|\sum_{i=1}^{n}\lambda_i(T-T_{\alpha(i)})\right\| = \operatorname{max}\left\{\|T\|-1,\|T\|_e\right\}-\varepsilon_n. 
\end{equation}
Let $S = \sum_{i=1}^{n}\lambda_i(T-T_{\alpha(i)})$. Then, by Lemma~\ref{2}, there exists an ${\alpha(n+1)}>\alpha{(n)}$ such that 
\begin{equation}\label{10}
\left\|S+\beta(T-T_{\alpha(n+1)})\right\| \leq \varepsilon_{n+1}+ \operatorname{max}\left\{\left\|  S +\frac {\beta(\|T\|-1)T}{\|T\|}\right\|, \left\| S \right\|_e+\beta\|T\|_e\right\}  \mbox{ for all } \beta \in [0, 1].
\end{equation}
Now, to obtain $\lambda_{n+1}$, we consider the quantity:
\begin{equation}\label{eq2}
\left\|S+\beta(T-T_{\alpha(n+1)})\right\| = \left\|\sum_{i=1}^{n}\lambda_i(T-T_{\alpha(i)})+\beta(T-T_{\alpha(n+1)})\right\|.
\end{equation}
For $\beta = 0$, by the induction hypothesis (\ref{ind12}), the quantity (\ref{eq2}) becomes $\operatorname{max}\left\{\|T\|-1,\|T\|_e\right\}-\varepsilon_n$. As  $\beta\longrightarrow \infty$,  (\ref{eq2}) becomes larger than  $\operatorname{max}\left\{\|T\|-1,\|T\|_e\right\}-\varepsilon_{n+1}$. Hence there exists a $\lambda_{n+1}$ such that  
\begin{equation}\label{13}
\left\|\sum_{i=1}^{n+1}\lambda_i(T-T_{\alpha(i)})\right\| = \operatorname{max}\left\{\|T\|-1,\|T\|_e\right\}-\varepsilon_{n+1}. 
\end{equation} 

Then, from (\ref{13}), we obtain
$$ \operatorname{max}\left\{\|T\|-1,\|T\|_e\right\}-\varepsilon_{n+1} = \left\|\sum_{i=1}^{n+1}\lambda_i(T-T_{\alpha(i)})\right\| =  \left\|(\sum_{i=1}^{n+1}\lambda_i)T-\sum_{i=1}^{n+1}\lambda_iT_{\alpha(i)}\right\|\geq (\sum_{i=1}^{n+1}\lambda_i)\operatorname{max}\left\{\|T\|-1,\|T\|_e\right\}.$$ 
Thus  $\sum_{i=1}^{n+1}\lambda_i \leq 1$. 

Now take $\beta = \lambda_{n+1}$ in (\ref{10}). Then we have
\begin{align*}
\left\|\sum_{i=1}^{n+1}\lambda_i(T-T_{\alpha(i)})\right\| &= \left\|\sum_{i=1}^{n}\lambda_i(T-T_{\alpha(i)})+\lambda_{n+1}(T-T_{\alpha(n+1)})\right\|\\
&\leq \varepsilon_{n+1}+ \operatorname{max}\left\{\left\|  \sum_{i=1}^{n}\lambda_i(T-T_{\alpha(i)}) +\frac {\lambda_{n+1}(\|T\|-1)T}{\|T\|}\right\|, \left\| \sum_{i=1}^{n}\lambda_i(T-T_{\alpha(i)}) \right\|_e+\lambda_{n+1}\|T\|_e\right\} \\
&\leq \varepsilon_{n+1}+ \operatorname{max}\left\{\left\|  \sum_{i=1}^{n}\lambda_i(T-T_{\alpha(i)}) +\frac {\lambda_{n+1}(\|T\|-1)T}{\|T\|}\right\|,  (\sum_{i=1}^{n}\lambda_i)\left\| T\right\|_e+\lambda_{n+1}\|T\|_e\right\}\\
&\leq \varepsilon_{n+1}+ \operatorname{max}\left\{\left\|  \sum_{i=1}^{n}\lambda_i(T-T_{\alpha(i)}) +\frac {\lambda_{n+1}(\|T\|-1)T}{\|T\|}\right\|,  (\sum_{i=1}^{n+1}\lambda_i)\left\| T\right\|_e\right\}.
\end{align*}

Hence, by induction, we obtain a sequence $(\lambda_i)$ of positive numbers and a sequence $(T_{\alpha(i)})$ of compact operators  satisfying (\ref{lambda}), (\ref{12}), (\ref{b}), (\ref{ineqaulitymain}). 

Since $\sum_{i=1}^{\infty}\lambda_i \leq 1$, it is easy to see that the series  $\sum_{i=1}^{\infty}\lambda_iT_{\alpha(i)}$ converges.  Let $K=\sum_{i=1}^{\infty}\lambda_iT_{\alpha(i)}$. Then $K$ is a compact operator and $\|K\|\leq 1$.

From  (\ref{lambda}) and  (\ref{12}), it follows that  
\begin{equation}\label{max}
\left\|\sum_{i=1}^{\infty}\lambda_i(T-T_{\alpha(i)})\right\| =  \operatorname{max}\left\{\left\|T\right\|-1, \|T\|_e\right\}.
\end{equation}

If $\sum_{i=1}^{\infty}\lambda_i = 1$, then from (\ref{max}) and using the fact that $\operatorname{max}\left\{\left\|T\right\|-1, \|T\|_e\right\}\leq d(T, B_{\mathcal{K}(X)}) $, we get 
$$d(T, B_{\mathcal{K}(X)})\leq \left\| T-K\right\|=\left\|\sum_{i=1}^{\infty}\lambda_i(T-T_{\alpha(i)})\right\|= \operatorname{max}\left\{\left\|T\right\|-1, \|T\|_e\right\}\leq d(T, B_{\mathcal{K}(X)}).$$
Hence $K$ is a best approximation to $T$ from the closed unit ball of $\mathcal{K}(X)$.

So it is enough to prove that $\sum_{i=1}^{\infty}\lambda_i \geq 1$. To obtain this, we split the proof into two cases depending on where the maximum occurs in the inequality (\ref{ineqaulitymain}).

{\bf Case 1:}   Assume that the maximum in the inequality (\ref{ineqaulitymain}) is attained at the second term for infinitely many $n$. That is, there exists an increasing sequence $(n_k)$ of positive integers such that 
$\left\|\sum_{i=1}^{n_k+1}\lambda_i(T-T_{\alpha(i)})\right\| \leq \varepsilon_{n_k+1}+ (\sum_{i=1}^{n_k+1}\lambda_i)\left\| T\right\|_e \leq \varepsilon_{n_k+1}+ (\sum_{i=1}^{\infty}\lambda_i)\left\| T\right\|_e $.
Then $$\operatorname{max}\left\{\|T\|-1, \|T\|_e\right\} =  \lim_{k\rightarrow\infty} \left\|  \sum_{i=1}^{n_k}\lambda_i(T-T_{\alpha(i)})\right\| \leq \lim_{k\rightarrow\infty}   \varepsilon_{n_k}+ (\sum_{i=1}^{\infty}\lambda_i)\left\| T\right\|_e,$$
Hence $\operatorname{max}\left\{\|T\|-1, \|T\|_e\right\}\leq (\sum_{i=1}^{\infty}\lambda_i)\left\| T\right\|_e$ and therefore $(\sum_{i=1}^{\infty}\lambda_i)\geq 1.$

{\bf Case 2:}  Assume that the maximum in the inequality  (\ref{ineqaulitymain}) is attained at  the second term only for finitely many $n$. 

We prove the case when the  maximum  in the inequality  (\ref{ineqaulitymain}) is attained at the  second term for at least one $n\in \mathbb{N}$ or maximum  in the inequality (\ref{b}) is attained at the  second term. A similar proof holds  even if the  maximum  in both inequalities  (\ref{b}) and  (\ref{ineqaulitymain}) is not attained at the  second term for any $n\in \mathbb{N}$. Hence  we assume that there exists a positive integer $N$ such that
$$\left\|\sum_{i=1}^{N}\lambda_i(T-T_{\alpha(i)})\right\| \leq \varepsilon_{N}+(\sum_{i=1}^{N }\lambda_i)\left\| T\right\|_e,\quad \left\|\sum_{i=1}^{m}\lambda_i(T-T_{\alpha(i)})\right\| \leq \varepsilon_{m}+ \left\|  \sum_{i=1}^{m-1}\lambda_i(T-T_{\alpha(i)}) +\frac {\lambda_{m}(\|T\|-1)T}{\|T\|}\right\|$$ for all $m>N.$ Now for any $m>N$, 
\begin{align*}
\varepsilon_{m}+\left\|\sum_{i=1}^{m-1}\lambda_i(T-T_{\alpha(i)}) +\frac {\lambda_{m}(\|T\|-1)T}{\|T\|}\right\|& \leq \varepsilon_{m}+ \left\|\sum_{i=1}^{m-1}\lambda_i(T-T_{\alpha(i)})\right\|+\left\|\frac{\lambda_{m}(\|T\|-1)T}{\|T\|}\right\|\\
&\leq \varepsilon_{m}+\left\|\sum_{i=1}^{m-1}\lambda_i(T-T_{\alpha(i)})\right\|+\lambda_{m}(\|T\|-1)\\
&\leq \left\|\sum_{i=1}^{m-k=N}\lambda_i(T-T_{\alpha(i)})\right\|+\sum_{i=N+1}^{m}\lambda_{i}(\|T\|-1)+\sum_{i=N+1}^{m}\varepsilon_{i}.
\end{align*}  
Now using (\ref{max}) and letting $m \to \infty$ in the above inequality, we can see that
\begin{align}
\operatorname{max}\left\{\left\|T\right\|-1, \|T\|_e\right\} 
= \left\|\sum_{i=1}^{\infty}\lambda_i(T-T_{\alpha(i)})\right\|&= \lim_{m\to \infty}\left\|\sum_{i=1}^{m}\lambda_i(T-T_{\alpha(i)})\right\|\nonumber\\
&\leq \left\|\sum_{i=1}^{N}\lambda_i(T-T_{\alpha(i)})\right\| + \sum_{i=N+1}^{\infty}\lambda_i(\|T\|-1)+\sum_{i=N+1}^{\infty}\varepsilon_i\nonumber\\
&\leq \varepsilon_{N}+(\sum_{i=1}^{N}\lambda_i)\left\| T\right\|_e +  \sum_{i=N+1}^{\infty}\lambda_i(\|T\|-1)+\sum_{i=N+1}^{\infty}\varepsilon_i \nonumber\\
&\leq  \sum_{i=1}^{\infty}\lambda_i\operatorname{max}\left\{\left\|T\right\|-1, \|T\|_e\right\}+\sum_{i=1}^{\infty}\varepsilon_i.\label{similar}
\end{align}
Thus  $a \leq \frac{\operatorname{max}\left\{\left\|T\right\|-1,\|T\|_e\right\}-\sum_{i=1}^{\infty}\varepsilon_i}{\operatorname{max}\left\{\left\|T\right\|-1,\|T\|_e\right\}}\leq \sum_{i=1}^{\infty}\lambda_i\leq 1 $. 

Now we will prove that $ \sum_{i=1}^{\infty}\lambda_i \geq 1$.  Let $b =\sum_{i=1}^{\infty}\lambda_i$. 

Suppose $b<1$. Then, by (\ref{max}),  $\left\|bT-K\right\| =  \operatorname{max}\left\{\left\|T\right\|-1,\|T\|_e\right\}$.  
Take $t =1- \frac{b\|T\|-1}{\|T\|}$. Then  $0<t<1$ and
$\left\|bT-\left(tbT+(1-t)K\right)\right\|=(1-t)\operatorname{max}\left\{\left\|T\right\|-1,\|T\|_e\right\}\leq (1-t)\|T\| \leq b\|T\| -1 .$ 
Now take $\lambda = \frac{1}{tb\|T\|+1}$. Since $0<\lambda < 1 $ and $\|bT- \frac{T}{\|T\|}\|=b\|T\|-1$, we have 
\begin{align*}
&\left\|bT-\left(\lambda \left(tbT+\left(1-t\right)K\right)+\left(1-\lambda\right)\frac{T}{\|T\|}\right)\right\| \leq b\|T\| -1. 
\end{align*}
Since $\lambda \left(tbT+\left(1-t\right)K\right)+\left(1-\lambda\right)\frac{T}{\|T\|}=\left(\lambda tb-\frac{(1-\lambda)}{\|T\|}\right)T+\lambda\left(1-t\right)K=\lambda(1-t)K$, it follows that 
$ \left\|bT-\lambda(1-t)K\right\| \leq b\|T\|-1$. Thus $\lambda(1-t)K$ is a best approximation to $bT$ from the closed unit ball of $\mathcal{B}(X)$ and $\|\lambda(1-t)K\| < b <1$, which is a contradiction as a best approximation to $bT$ from the closed unit ball of $\mathcal{B}(X)$  always has norm $1$. Hence $b=\sum_{i=1}^{\infty}\lambda_i = 1$.
 
 Now let $T\in \mathcal{B}(X) $  be such that $\|T\| \leq 1$. Then, by \cite{MR1238713}*{Chapter~\Rmnum{6}, Proposition 4.10}, there exists a net $(K_\alpha)$ of compact operators  with $\|{K_\alpha}\| \leq 1 $ and  $(K_\alpha^*)$ converges to $T^*$ in SOT. Then, by \cite{MR1257062}*{Corollary 3.2}, there exists a compact operator  $ K \in \overline{\conv}\{K_\alpha \}$ such that  $\|T-K\| = d(T, \mathcal{K}(X))$, where $\overline{\conv}\{K_\alpha\}$ denotes the norm closure of $\conv\{K_\alpha\}$. Since $\|K_{\alpha}\|\leq 1 $, we get $\|K\| \leq 1$ and $d(T, B_{\mathcal{K}(X)}) = d(T, {\mathcal{K}(X)}) =\|T-K\|$. Hence $\mathcal{K}(X)$ is ball proximinal in $\mathcal{B}(X)$.
\end{proof}

We now state a remark which follows from \cite{MR1971228}*{Theorem~1}.
\begin{rmk}
Let $H$ be a separable Hilbert space. Then for each positive (self-adjoint) operator $T$ on $H$, there exists a compact positive (self-adjoint)  operator $K$ such that $d(T, \mathcal{K}(H))=\|T-K\|$. For, let	$\{e_1,e_2,...\}$ be an orthonormal basis for $H$ and $P_n$ be the orthogonal projection onto $\{e_1,e_2,...,e_n\}$. Put $T_n=P_nTP_n$ for $n\ge 1$. Then $(T_n)$ is a sequence of compact positive (self-adjoint)  operator such that $T_n \rightarrow T$ in SOT and $T_n^*\rightarrow T^*$ in SOT. Now, by \cite{MR1971228}*{Theorem~1}, there exists a sequence $(a_n)$ of non-negative real numbers such that $K=\sum_{n=1}^\infty a_n T_n $ is a compact positive (self-adjoint) operator that satisfies $d(T, \mathcal{K}(H))=\|T-K\|$.
\end{rmk}

The following result shows that positive (self-adjoint) operators on a Hilbert space $H$ has a positive (self-adjoint) compact approximant from the closed unit ball of $\mathcal{K}(H)$. 
\begin{cor}
	Let $H$ be a separable Hilbert space. Then for each positive (self-adjoint) operator $T$ on $H$ with $\|T\|\geq 1$, there exists a positive (self-adjoint) compact operator $K$ on $H$ such that $\|K\|\leq 1$ and $d(T, B_{\mathcal{K}(H)}) = \|T-K\|.$ 
\end{cor}
\begin{proof}
Let $\{e_1,e_2,...\}$ be an orthonormal basis for $H$ and $P_n$ be the orthogonal projection onto $\{e_1,e_2,...,e_n\}$ for $n\ge 1$. Put $K_n=P_n\frac{T}{\|T\|}P_n$ for $n\ge 1$. Then $(K_n)$ is a sequence of positive  (self-adjoint) compact operators such that $\|K_n\|\leq 1$, $K_n \rightarrow \frac{T}{\|T\|}$ in SOT and $K_n^*\rightarrow \frac{T^*}{\|T\|}$ in SOT. Now, by proceeding as in the proof  of  Theorem~\ref{thmmain}, we get  a sequence $(\lambda_n)$ of non-negative real numbers  such that $K=\sum_{n=1}^\infty \lambda_n T_n $ is a compact positive (self-adjoint) operator   that satisfies $\|K\|\leq 1$ and $\|T-K\|= d(T, B_{\mathcal{K}(H)}),$  where $T_n\in \conv\{K_m : m\geq 1\}.$
\end{proof}

 We observe from the proof of Theorem \ref{thmmain} that if $Y$ is an $M$-ideal in $X$ and if there is a net $(y_\alpha)$ such that $\|y_\alpha\| \leq 1$ and $(y_\alpha)$ satisfies the inequality (\ref{genineq}) of Lemma~\ref{lemmma}, then $Y$ is ball proximinal in $X$ and for $x\in X$, $d(x, B_Y) = \operatorname{max}\{\|x\|-1, d(x, Y)\}$. We now use this fact to prove that $M$-embedded spaces are ball proximinal in its bidual.
\begin{thm}
Let $X$ be an $M$-embedded space. Then $X$ is ball proximinal in $X^{**}$ and  $d(x^{**}, B_X) = \operatorname{max}\{\|x^{**}\|-1,d(x^{**},X) \}$ for $x^{**}\in X^{**}$. 
\end{thm}
\begin{proof}
Let $x^{**}\in X^{**}$ be such that $\|x^{**}\|> 1$ and $(y_\alpha)$  be the net in $X$ devised from Lemma \ref{lemmma}. Since $B_X$ is weak$^*$-dense in $B_{X^{**}}$, there exists a net $(x_\alpha)$ in $X$ such that $\|x_\alpha\| \leq 1$ and $(x_\alpha)$ converges to $\frac{x^{**}}{\|x^{**}\|}$ in the weak*-topology. Then the net  $( x_\alpha - y_\alpha)$ converges to 0 in the weak*- topology of $X^{**}$ and hence $(x_\alpha - y_\alpha)$ converges to 0 in the weak topology of $X$. Therefore $\|{z_\alpha}-{u_\alpha}\| \to 0$ for some ${z_\alpha}\in  \conv\{x_\beta : \beta\geq \alpha\}$ and ${u_\alpha}\in  \conv\{y_\beta : \beta\geq \alpha\}$. Clearly $\|{z_\alpha}\|\leq 1$. Since $({u_\alpha})$ satisfies inequality~(\ref{genineq}), it follows that the net $({z_\alpha})$  satisfies inequality~(\ref{genineq}). Now by proceeding as in the proof of  of Theorem \ref{thmmain}, we can construct a  sequence $(\lambda_n)$ of non-negative real numbers and a sequence $(z_{\alpha_n})$ from the net $({z_\alpha})$ such that $z=\sum_{n=1}^\infty \lambda_n z_{\alpha_n}$ is a best approximation to $x^{**}$ from $B_X$ and $d(x^{**}, B_X) = \operatorname{max}\{\|x^{**}\|-1,d(x^{**},X) \}=\|x^{**}-z\|$.

If $\|x^{**}\|\leq 1$, then there exists a net $(x_\alpha)$ in $B_X$ such that $(x_\alpha)$ converges to $x^{**}$ in the weak* topology. Since the weak* topology on $X^{**}$ coincide with $\sigma(X^{**}, X^*)$-topology, by the remark following \cite{MR1257062}*{Corollary 3.2}, there exists an element $y\in \overline{\conv}\{x_\alpha\}$ such that  $d(x^{**}, X) = \|x-y\|$. Since $\|x_\alpha\| \leq 1$ for all $\alpha$, we have $y\in B_X$ and  hence  $d(x^{**}, B_X) = d(x^{**},X) = \|x-y\|$. Therefore $X$ is ball proximinal in $X^{**}$.
\end{proof}

\begin{rmk}
 If $X$ is a reflexive space such that $\mathcal{K}(X)$ is an $M$-ideal in $\mathcal{B}(X)$, then, by \cite{MR1238713}*{Chapter~\Rmnum{6}, Proposition~4.11}, $\mathcal{B}(X)$ is isometric to the bidual of $\mathcal{K}(X)$ and hence, by Theorem~\ref{thmmain}, $\mathcal{K}(X)$ is ball proximinal in its bidual.
\end{rmk}

 We now give an example of a Banach space $X$ such that $\mathcal{K}(X)$ is ball proximinal in $\mathcal{B}(X)$, but $\mathcal{K}(X)$ is not an $M$-ideal in $\mathcal{B}(X)$. It is well known that $\mathcal{K}(\ell_1)$ is not an $M$-ideal in $\mathcal{B}(\ell_1)$ (see \cite{MR1238713}). 
	
	\begin{eg}
		$\mathcal{K}(\ell_1)$ is ball proximinal in $\mathcal{B}(\ell_1)$ and  $d(T, B_{\mathcal{K}(\ell_1)})= \operatorname{max}\{\|T\|-1, d(T, {\mathcal{K}(\ell_1)})\}$ for $T\in \mathcal{B}(\ell_1)$.
	\end{eg}
	
\begin{proof}
	We know from \cite{MR0098966}*{Page 220} that every operator $T$ on $\ell_1$ has a matrix representation $[t_{i j}]$ with respect to the canonical basis $\{e_i\}_1^\infty$.

	Let $T=[t_{ij}] \in \mathcal{B}(\ell_1)$. Then, by \cite{MR505751}*{Theorem 1.3}, $d(T, {\mathcal{K}(\ell_1)})= R$, where $R= \lim _{n \rightarrow \infty} \sup _{j} \sum_{i=n}^{\infty}\left|t_{i j}\right|$. 
    Now let $d = \max\{\|T\|-1, R\}$. Clearly, $d\leq d(T, B_{\mathcal{K}(\ell_1)})$. We follow a similar construction as in \cite{MR505751}*{Theorem 1.3} to obtain a compact operator $K = [k_{ij}]$ such that $\|T-K\|=d(T, B_{\mathcal{K}(\ell_1)})= d$ and $\|K\|\leq1$. 
    
    For a  fixed $j\in \mathbb{N}$, the $j^{th}$ column of $K=[k_{ij}]$ is defined as follows:
    
     If $ \sum_{i=1}^{\infty}\left|t_{i j}\right| \leqslant d$, set $k_{i j}=0$ for $i=1, \ldots, \infty$. 
     
      If $ \sum_{i=1}^{\infty}\left|t_{i j}\right|>d,$  let $n$ be the largest index such that  $\sum_{i=n}^\infty|t_{i j}|> d$. Then there exists a real number $a \in [0, 1]$ such that $ a|t_{n j}|+\sum_{i=n+1}^\infty|t_{i j}|  = d$. Now set $k_{i j}$ as:  
      $$
      k_{i j} = 
      \begin{cases}
      {t_{i j}}& \text{if}\quad1\leq i<n,\\
      {(1-a)t_{i j}}&\text{if}\quad i= n,\\
      0 & \text{if} \quad i> n.
      \end{cases}
$$
    Now to prove $\|K\|\leq 1$, we assume $n>1$ (a similar proof works for $n=1$). Then for $j$ with $ \sum_{i=1}^{\infty}\left|t_{i j}\right|>d$, since 
$\sum_{i=1}^{n-1}|t_{i j}|+a|t_{n j}|+(1-a)|t_{n j}|+\sum_{i=n+1}^\infty|t_{i j}|= \sum_{i=1}^\infty|t_{i j}|\leq \|T\|$ and $d \geq \|T\|-1$, we get 
    $$ \sum_{i=1}^{\infty}|k_{i j}|=\sum_{i=1}^{n-1}|t_{i j}|+(1-a)|t_{n j}|\leq \|T\|-d\leq 1.$$
     
 Therefore $\|K\|\leq 1$ and $\|T-K\|=\sup _{j} \sum_{i=1}^{\infty}|t_{ij}-k_{ij}|=d$. Now  a similar argument as in the proof of \cite{MR505751}*{Theorem 1.3} gives the compactness of the operator $K$. Thus $d(T,B_{\mathcal{K}(\ell_1)}) = \|T-K\| = \operatorname{max}\{\|T\|-1, d(T, {\mathcal{K}(\ell_1)}) \}$ and hence $\mathcal{K}(\ell_1)$ is ball proximinal in $\mathcal{B}(\ell_1)$.  
\end{proof}
\section*{Acknowledgements}
The research of the first named author is supported by SERB MATRICS grant(No. MTR/2017/000926) and the research of the second named author is supported by UGC Junior research fellowship (No. 20/12/2015(ii)EU-V).

\end{document}